\documentclass[11pt]{article}
\usepackage[a4paper,margin=1in]{geometry}
\usepackage[T1]{fontenc}
\usepackage[utf8]{inputenc}
\usepackage{lmodern}
\usepackage{microtype}
\usepackage{amsmath,amssymb,amsthm,mathtools}
\usepackage{enumitem}
\usepackage{hyperref}
\usepackage[nameinlink,noabbrev]{cleveref}
\usepackage[
  backend=biber,
  url=false,
  giveninits=true
]{biblatex}

\DeclareFieldFormat[article]{title}{#1}
\DeclareFieldFormat[inproceedings]{title}{#1}
\DeclareFieldFormat[incollection]{title}{#1}
\DeclareFieldFormat[inbook]{title}{#1}

\renewbibmacro{in:}{}

\DeclareNameAlias{author}{family-given}
\DeclareNameAlias{editor}{family-given}
\DeclareNameAlias{translator}{family-given}
\DeclareNameAlias{sortname}{family-given}
\DeclareNameAlias{default}{family-given}

\AtEveryBibitem{%
  \clearfield{issn}%
  \clearfield{url}%
  \clearfield{urlyear}%
  \clearfield{urlmonth}%
  \clearfield{urlday}%
}

\hypersetup{colorlinks=true,linkcolor=blue,citecolor=blue,urlcolor=blue}

\newtheorem{theorem}{Theorem}[section]
\newtheorem{proposition}[theorem]{Proposition}
\newtheorem{lemma}[theorem]{Lemma}

\theoremstyle{definition}
\newtheorem{definition}[theorem]{Definition}
\newtheorem{remark}[theorem]{Remark}
\newtheorem{notation}[theorem]{Notation}

\newcommand{\R}{\mathbb{R}}

\newcommand{\absconv}{\operatorname{absconv}}
\newcommand{\VCdim}{\operatorname{VCdim}}

\newcommand{\bx}{b_X}
\newcommand{\eps}{\varepsilon}

\newcommand{\doi}[1]{\href{https://doi.org/#1}{doi:\nolinkurl{#1}}}
\newcommand{\norm}[1]{\left\lVert #1 \right\rVert}
\newcommand{\ip}[2]{\left\langle #1,#2 \right\rangle}

\title{Failure of Convex-Hull Bounds under Log-Convex Tails}

\author{
Xuanang Hu\thanks{Shandong University, Jinan 250100, China. 
Email: xuananghu7@gmail.com}
\and
Hanchao Wang\thanks{Shandong University, Jinan 250100, China. 
Email: wanghanchao@sdu.edu.cn}
}

\date{}

\begin{document}
\maketitle

\begin{abstract}
Fix $0<r<1$, and let $X_1,X_2,\dots$ be independent symmetric Weibull$(r)$ random variables, that is,
\[
\textsf{P}(|X_i|>t)=e^{-t^r},\qquad t\ge 0.
\]
We prove that there is no constant $C_r$, depending only on $r$, with the following universal property: for every finite set $T\subset \R^N$ there exists a sequence $(y_k)_{k\ge 1}\subset \R^N$ such that
\[
T-T\subset conv\{y_k:k\ge 1\},
\qquad
\|X_{y_k}\|_{L_{\log(k+2)}}\le C_r\,\bx(T)
\quad (k\ge 1),
\]
where $X_t=\sum_i t_i X_i$ and $\bx(T)=\textsf{E}\sup_{t\in T}X_t$. This gives a negative answer to a question of Lata{\l}a concerning the validity
of convex-hull bounds for canonical Weibull processes. In fact, the failure
persists even when the auxiliary vectors appearing in the convex hull are
allowed to be arbitrary. \\



\end{abstract}

\section{Introduction and statement of the main results}

Let \(X=(X_1,\ldots,X_N)\) be a random vector in \(\mathbb R^N\) whose
coordinates are independent and centered. For \(t=(t_1,\ldots,t_N)\in
\mathbb R^N\), set
\[
        X_t:=\langle t,X\rangle=\sum_{i=1}^N t_iX_i .
\]
Given an index set \(T\subset \mathbb R^N\), we consider the associated
canonical process \((X_t)_{t\in T}\) and denote its expected supremum by
\[
        b_X(T):=\textsf E \sup_{t\in T} X_t ,
\]
whenever this quantity is finite. Estimating \(b_X(T)\) in terms of the
geometry of \(T\) and the distributional properties of \(X\) is a central
problem in the theory of canonical processes.

In the Gaussian case, namely when \(X=G\) has independent standard
Gaussian coordinates, Talagrand's majorizing measure theorem
\cite{Talagrand1996GenericChaining,Talagrand2021UpperLower} gives optimal two-sided bounds in terms of the generic chaining
functional. This result is one of
the cornerstones of the modern theory of Gaussian processes and generic
chaining.

Beyond the Gaussian setting, the theory is more delicate and depends
strongly on the tail behavior of the coordinates. Lata{\l}a
\cite{Latala1997Sudakov} studied the canonical processes generated by
independent random variables with log-concave tails, while Lata{\l}a and
Tkocz \cite{LatalaTkocz2015RegularMoments} developed a framework for
variables with regularly growing moments. The Bernoulli case exhibits a
different behavior, and the boundedness problem for Bernoulli processes was
solved by Bednorz and Lata{\l}a
\cite{BednorzLatala2014Bernoulli} via a decomposition theorem.

In addition to generic chaining, Lata{\l}a introduced a convex-hull
approach tailored to certain classes of heavier-tailed canonical processes;
see \cite{Latala2023}. For clarity, we first recall the chaining
formulation and then explain how it naturally leads to convex-hull bounds.

We shall use the following comparison notation throughout the paper. The
symbols \(C,c>0\) denote positive universal constants whose values may
change from line to line. For two nonnegative quantities \(A\) and \(B\),
we write \(A\lesssim B\) if \(A\le CB\), and \(A\simeq B\) if both
\(A\lesssim B\) and \(B\lesssim A\) hold. The subscripts indicate the
parameters on which the implicit constants are allowed to depend; for
example, \(A\lesssim_r B\) means that \(A\le C(r)B\) and
\(A\simeq_r B\) means that both \(A\lesssim_r B\) and
\(B\lesssim_r A\) hold. 

We shall also distinguish between probabilistic and deterministic
$p$-norms. If $Z$ is a real-valued random variable and $p\ge 1$, then
\[
        \norm{Z}_{L_p}:=(\textsf E |Z|^p)^{1/p}
\]
denotes the $L_p$-norm of $Z$ on the underlying probability space. Thus,
for $u\in\mathbb R^N$, the quantity $\norm{X_u}_{L_p}$ is the $p$-th
moment norm of the random linear form $X_u=\sum_i u_iX_i$. In contrast,
\[
        \norm{u}_p:=\Bigl(\sum_{i=1}^N |u_i|^p\Bigr)^{1/p},
        \qquad
        \norm{u}_\infty:=\max_{1\le i\le N}|u_i|,
\]
are deterministic $\ell_p$-norms of the coefficient vector $u$. In
particular, whenever the subscript $L_p$ is used, the norm is taken with
respect to probability; without the letter $L$, the norm is taken in the
coordinate space.

In the generic chaining framework, one considers the family of
pseudo-metrics
\[
        d_p(s,t):=\|X_s-X_t\|_{L_p}, \qquad p\ge 1 .
\]
An admissible sequence of partitions of \(T\) is a refining sequence
\((\mathcal A_n)_{n\ge0}\) such that \(\mathcal A_0=\{T\}\) and
\[
        |\mathcal A_n|\le 2^{2^n}, \qquad n\ge 1 .
\] One defines
\[
        \gamma_X(T):=
        \inf_{\mathcal A}
        \sup_{t\in T}\sum_{n\ge0}
        \Delta_{2^n,X}(A_n(t)),
\]
where
\[
        \Delta_{p,X}(A):=\sup_{s,t\in A}\|X_s-X_t\|_{L_p}.
\]
Here, \( A_n(t)\) denotes the element of \(\mathcal A_n\)
containing \(t\).  Under suitable regularity assumptions on the moments of
the coordinates, see Lata{\l}a and Tkocz
\cite{LatalaTkocz2015RegularMoments}, one has
\[
        b_X(T)\simeq \gamma_X(T).
\]

We now recall how a convex-hull bound arises from the chaining
construction.  Assume for simplicity that \(T\) is finite, and let
\((\mathcal A_n)\) be an admissible sequence of partitions which nearly
attains \(\gamma_X(T)\). For each  \(A\in \mathcal A_n\), choose a
point \(t_A\in A\). For \(t\in T\), set
\[
        \pi_n(t):=t_{A_n(t)} .
\]

Since \(T\) is finite, the partitions may be chosen so that they separate
all points at sufficiently large levels. Hence \(\pi_n(t)=t\) for all large
\(n\), and therefore
\[
        t-\pi_0(t)
        =
        \sum_{n\ge1}\bigl(\pi_n(t)-\pi_{n-1}(t)\bigr).
\]
Because the partitions are refining, both \(\pi_n(t)\) and
\(\pi_{n-1}(t)\) belong to \( A_{n-1}(t)\).

The increments in the last display are controlled by the diameters of
\( A_{n-1}(t)\):
\[
        \|X_{\pi_n(t)-\pi_{n-1}(t)}\|_{L_{2^n}}
        \le
        \Delta_{2^n,X}( A_{n-1}(t)).
\]
Moreover, the number of possible increments at level \(n\) is at most of
order \(2^{2^n}\).  After listing these increments with ranks
\(k\), the relation \(2^n\simeq \log k\) gives vectors
\((y_k)_{k\ge1}\) such that
\[
        T-T\subset \operatorname{conv}\{y_k:k\ge1\},
        \qquad
        \|X_{y_k}\|_{L_{\log(k+2)}}\lesssim \gamma_X(T).
\]
Consequently, for any class of canonical processes for which
\(\gamma_X(T)\lesssim b_X(T)\), the chaining estimate implies a
convex-hull bound. More precisely, one can find vectors
\((y_k)_{k\ge1}\subset\mathbb R^N\) such that
\[
        T-T\subset \operatorname{conv}\{\pm y_k:k\ge1\},
        \qquad
        \|X_{y_k}\|_{L_{\log(k+2)}}\lesssim b_X(T),
        \quad k\ge1 .
\]

In the final section of \cite{Latala2023}, Lata{\l}a discussed the canonical
Weibull processes and pointed out that Bogucki had obtained two-sided
bounds for \(b_X(T)\) using random permutations in \cite{Bogucki2015}. However, he noted that it was not known
whether the convex-hull method works in this setting. Thus, the validity of
a convex-hull bound for canonical Weibull processes was left as an open
problem.

Canonical Weibull processes constitute a natural class of canonical
processes generated by independent coordinates with log-convex tails. We
begin by recalling their definition.

\begin{definition}\label{def:weibull-process}
Fix \(0<r<1\).  A random variable \(\xi\) is called symmetric
Weibull\((r)\) if \(\xi\) is symmetric and
\[
        \textsf P(|\xi|>t)=\exp(-t^r),\qquad t\ge 0 .
\]
Given independent copies \(X_1,\ldots,X_N\) of \(\xi\), we consider the
canonical process
\[
        X_t=\sum_{i=1}^N t_iX_i,\qquad t\in\mathbb R^N .
\]
For a finite set \(T\subset\mathbb R^N\), we keep the notation
\[
        b_X(T)=\textsf{E}\sup_{t\in T}X_t .
\]
\end{definition}

From now on we focus on canonical Weibull processes with \(0<r<1\).
Their moments satisfy
\[
        \|X_1\|_{L_p}\simeq_r p^{1/r},\qquad p\ge2.
\]
Thus their moments grow faster than linearly in \(p\), and hence these
variables fall outside the regular-moment framework described above.

Bogucki \cite{Bogucki2015} proved sharp two-sided estimates for the
expected suprema of canonical Weibull processes by using non-increasing
rearrangements. In the discussion of the convex-hull approach,
Lata{\l}a \cite{Latala2023} pointed out that these estimates can also be
viewed through the use of random permutations, which may be eliminated when
\(T\) is permutationally invariant. However, such estimates do not settle
whether the convex-hull method itself is valid for canonical Weibull
processes.

This leads to the open problem raised in \cite{Latala2023}: does every
finite set \(T\subset\mathbb R^N\) satisfy a convex-hull bound of the form
\[
        T-T\subset \operatorname{conv}\{\pm y_k:k\ge1\},
        \qquad
        \|X_{y_k}\|_{L_{\log(k+2)}}\lesssim_r b_X(T),
        \quad k\ge1,
\]
for some auxiliary vectors \((y_k)_{k\ge1}\subset\mathbb R^N\)?
For regular canonical processes, this bound follows from the
\(\gamma_X\)-functional. For canonical Weibull processes with \(0<r<1\),
the validity of this convex-hull bound remained open.

Our main result gives a negative answer to this problem, even in a weak
form. We allow the vectors appearing in the convex hull to be arbitrary
auxiliary vectors: they need not belong to \(T-T\), nor even be parallel to
differences of points of \(T\). Thus, obstruction does not come from a
restriction on directions, but rather from the \(L_{\log k}\)-moment
control itself.

The counterexample constructed in this paper is given in the specific setting of canonical Weibull processes. This is sufficient to disprove any convex-hull principle intended to hold uniformly for canonical processes with log-convex tails, since the Weibull case considered here forms a natural subclass of that framework. Thus, the failure identified below is not caused by an artificial choice of directions or by a boundary case of the theory; it already occurs in one of the basic log-convex-tail examples.

We are now ready to state our main result. For a set \(A\subset\mathbb R^N\), define its absolutely convex hull by
\[
\operatorname{absconv}(A)
 :=
 \left\{
     \sum_{j=1}^m a_j a^{(j)}:
     m<\infty,\ a^{(j)}\in A,\ \sum_{j=1}^m |a_j|\le1
 \right\}.
\]
Equivalently,
\[
        \operatorname{absconv}(A)=\operatorname{conv}(A\cup(-A)).
\]
If \(A=\{y_k:k\ge1\}\), then for every \(\theta\in\mathbb R^N\),
\begin{equation}\label{eq:support-absconv}
        \sup_{z\in\operatorname{absconv}\{y_k:k\ge1\}}
        \langle \theta,z\rangle
        =
        \sup_{k\ge1}|\langle \theta,y_k\rangle|.
\end{equation}
Indeed, every element of the absolutely convex hull is a finite
\(\ell_1\)-combination of the vectors \(y_k\), with total mass at most
one.

\begin{definition}\label{def:free-covering}
For $k\ge 1$, let
\[
p_k:=\max\{2,\log(k+2)\}.
\]
A sequence $(y_k)_{k\ge 1}\subset\R^N$ is called a \emph{free rank-wise covering sequence} for a finite set $T\subset\R^N$ if
\[
T-T\subset \absconv\{y_k:k\ge 1\}
\]
and there exists a constant $C>0$ such that
\[
\|X_{y_k}\|_{L_{\log(k+2)}}\le C\,\bx(T),\qquad k\ge 1.
\]
The adjective \emph{free} means that no condition is imposed on the provenance of $y_k$; in particular, $y_k$ need not belong to $T-T$, and need not even be parallel to a difference of points in $T$.
\end{definition}

\begin{theorem}\label{thm:main-impossibility}
Fix $0<r<1$. There does not exist a constant $C_r<\infty$, depending only on $r$, such that for every dimension $N$ and every finite set $T\subset\R^N$ one can find a sequence $(y_k)_{k\ge 1}\subset\R^N$ satisfying
\begin{equation}\label{eq:universal-false-statement}
T-T\subset \absconv\{y_k:k\ge 1\},
\qquad
\|X_{y_k}\|_{L_{\log(k+2)}}\le C_r\,\bx(T)
\quad (k\ge 1).
\end{equation}
\end{theorem}

\begin{remark}
The original problem raised by Lata{\l}a in \cite{Latala2023}  is stated in terms of the convex hull.
In this paper we work instead with the absolutely convex hull.
This change is only for convenience.

Indeed, in the present setting, the two formulations are
equivalent up to a constant depending only on \(r\).
Passing from the convex hull to the absolutely convex hull is immediate.
Conversely, it is enough to replace each vector \(y_k\) by the pair
\(y_k\) and \(-y_k\), and then arrange the resulting vectors as a sequence.
In this way, a statement formulated with the absolutely convex hull is reduced
to one formulated with the convex hull, with only a change in the constant.
A proof will be given after Lemma~\ref{lem:low-order-absorption}.
\end{remark}

We prove Theorem \ref{thm:main-impossibility} through the following quantitative statement.

\begin{notation}\label{not:gauge}
For $p\ge 2$ and $u\in\R^N$, set
\begin{equation}\label{eq:gauge-def}
\rho_p(u):=\sqrt{p}\,\norm{u}_2 + p^{1/r}\norm{u}_\infty.
\end{equation}
\end{notation}

\begin{theorem}\label{thm:quantitative-counterexample}
Fix $0<r<1$ and choose a parameter $a$ with
\[
1<a<\frac{2}{r}-1.
\]
Then for every fixed constant $A>0$ and all sufficiently large $Q$ there exist
\[
d=\lfloor Q^a\rfloor,
\qquad
M=\lfloor e^Q\rfloor,
\qquad
T=\{0,\eps^1,\dots,\eps^M\}\subset \{0\}\cup\{-1,1\}^d,
\]
such that
\begin{equation}\label{eq:bX-upper-main}
\bx(T)\le C_r\sqrt{dQ},
\end{equation}
but there is no sequence $(y_k)_{k\ge 1}\subset\R^d$ that simultaneously satisfies 
\begin{equation}\label{eq:T-covered}
T\subset \absconv\{y_k:k\ge 1\}
\end{equation}
and
\begin{equation}\label{eq:gauge-controlled}
\rho_{p_k}(y_k)\le A\sqrt{dQ},
\qquad p_k=\max\{2,\log(k+2)\},
\qquad k\ge 1.
\end{equation}
Consequently, no sequence satisfying \eqref{eq:gauge-controlled} can cover $T-T$ either.
\end{theorem}

\begin{remark}
Because $0\in T$, the inclusion $T-T\subset \absconv\{y_k\}$ implies $T\subset \absconv\{y_k\}$.
\end{remark}

We now explain how the set in Theorem \ref{thm:quantitative-counterexample} is constructed and how the proof is organized. The set \(T\) in
Theorem \ref{thm:quantitative-counterexample} is obtained by the probabilistic method.
For a large parameter \(Q\), we set
\[
d=\lfloor Q^a\rfloor,\qquad M=\lfloor e^Q\rfloor,
\]
and let
\[
\varepsilon^1,\ldots,\varepsilon^M
\]
be independent random vectors, each distributed uniformly on \(\{-1,1\}^d\). Equivalently,
the coordinates of each \(\varepsilon^j\) are independent symmetric signs. We consider the
random set
\[
T_\varepsilon:=\{0,\varepsilon^1,\ldots,\varepsilon^M\}.
\]

The proof has two parts. First, we show that \(T_\varepsilon\) has expected supremum of order
at most \(\sqrt{dQ}\):
\[
b_X(T_\varepsilon)\lesssim_r \sqrt{dQ}.
\]
This estimate holds with positive probability. Second, we show that, with probability tending
to one as \(Q\to\infty\), there is no sequence \((y_k)_{k\ge1}\subset\mathbb R^d\) such that
\[
T_\varepsilon\subset \operatorname{absconv}\{y_k:k\ge1\}
\]
and
\[
\rho_{p_k}(y_k)\le A\sqrt{dQ},\qquad
p_k=\max\{2,\log(k+2)\},\qquad k\ge1.
\]
These two probabilistic statements imply that, for all sufficiently large \(Q\), there is a
deterministic choice of the sign vectors for which both properties hold.

The obstruction to the absolutely convex representation is elementary but useful. Suppose that
\[
T_\varepsilon\subset \operatorname{absconv}\{y_k:k\ge1\}.
\]
Then, for every \(1\le j\le M\), applying \eqref{eq:support-absconv} with
\(\theta=\varepsilon^j\) gives
\[
d=\langle \varepsilon^j,\varepsilon^j\rangle
\le
\sup_{z\in \operatorname{absconv}\{y_k:k\ge1\}}
\langle \varepsilon^j,z\rangle
=
\sup_{k\ge1}|\langle \varepsilon^j,y_k\rangle|.
\]
Thus each point \(\varepsilon^j\) must satisfy
\[
\sup_{k\ge1}|\langle \varepsilon^j,y_k\rangle|\ge d.
\]
This motivates the definition
\[
U(y_\bullet):=
\left\{
\sigma\in\{-1,1\}^d:
\sup_{k\ge1}|\langle y_k,\sigma\rangle|\ge \frac d2
\right\}.
\]
Any sequence whose absolutely convex hull contains \(T_\varepsilon\) must therefore satisfy
\[
\varepsilon^j\in U(y_\bullet),\qquad 1\le j\le M.
\]

The main difficulty is that the sequence \((y_k)\) is not fixed in advance. It may depend on all
the random points \(\varepsilon^1,\ldots,\varepsilon^M\). We first prove that, for every fixed
sequence satisfying
\[
\rho_{p_k}(y_k)\le A\sqrt{dQ},\qquad k\ge1,
\]
the set \(U(y_\bullet)\) has small measure in the discrete cube. We then use a VC argument to
pass from fixed sequences to sequences chosen after the random points are given. The VC argument shows that, with probability tending to one, no sequence satisfying the above
bounds can have
\[
\varepsilon^j\in U(y_\bullet),\qquad 1\le j\le M.
\]

The rest of the paper follows this scheme. Section \(2\) collects the moment estimates and the
comparison between convex and absolutely convex hulls used throughout the proof. Section \(3\)
proves the expected supremum bound for the random set \(T_\varepsilon\). Section \(4\) shows
that, for each fixed sequence satisfying the bounds on \(\rho_{p_k}(y_k)\), the set
\(U(y_\bullet)\) has small measure in the discrete cube. Section \(5\) uses VC theory to show
that, with high probability, no sequence satisfying these bounds can have all the random points
\(\varepsilon^1,\ldots,\varepsilon^M\) contained in the corresponding set \(U(y_\bullet)\).
Section \(6\) chooses the deterministic set \(T\) and completes the proofs of Theorems
\ref{thm:main-impossibility} and \ref{thm:quantitative-counterexample}.

\section{Moment estimates and preliminary reductions}

The only non-elementary probabilistic input used in this paper is the
moment equivalence for sums of independent symmetric random variables with
log-convex tails. In the Weibull setting considered here, it takes the
following form.

\begin{theorem}\label{thm:HMO}
Let $X_1,\dots,X_N$ be independent symmetric Weibull$(r)$ random variables with $0<r\le 1$. Then there exist constants $0<c_r\le C_r<\infty$, depending only on $r$, such that for every $u\in\R^N$ and every $p\ge 2$,
\begin{equation}\label{eq:HMO}
c_r\,\rho_p(u)
\le
\norm{X_u}_{L_p}
\le
C_r\,\rho_p(u).
\end{equation}
\end{theorem}

\begin{proof}
    By a theorem of Hitczenko, Montgomery-Smith, and Oleszkiewicz~\cite{HMO1997},
see also Lata{\l}a~\cite{latalamoment}, we have
    \[
    \norm{X_u}_{L_p}\simeq_r \sqrt{p}\,\norm{u}_2 + p^{1/r}\norm{u}_p.
    \]
    Thus, we have
    \[
    \rho_p(u)
\lesssim_r
\norm{X_u}_{L_p}.
    \]

    To finish the proof, it is enough to show that
\[
\sqrt p\,\|t\|_2 + p^{1/r}\|t\|_p
\lesssim_r \bigl(\sqrt p\,\|t\|_2 + p^{1/r}\|t\|_\infty\bigr).
\]
By scaling, we may assume that
\[
\sqrt p\,\|t\|_2 + p^{1/r}\|t\|_\infty = 1.
\]
In particular,
\[
\|t\|_2 \le p^{-1/2},
\qquad
\|t\|_\infty \le p^{-1/r}.
\]

Now let \(p\ge 2\). Since
\[
|t_k|^p = |t_k|^2 |t_k|^{p-2}
\le |t_k|^2 \|t\|_\infty^{\,p-2},
\]
we get
\[
\|t\|_p
=
\Bigl(\sum_k |t_k|^p\Bigr)^{1/p}
\le
\Bigl(\sum_k |t_k|^2 \|t\|_\infty^{\,p-2}\Bigr)^{1/p}
=
\bigl(\|t\|_2^2 \|t\|_\infty^{\,p-2}\bigr)^{1/p}.
\]
Using the bounds on \(\|t\|_2\) and \(\|t\|_\infty\), we obtain
\[
\|t\|_p
\le
\bigl(p^{-1} p^{-(p-2)/r}\bigr)^{1/p}
=
p^{-1/r} p^{(2-r)/(rp)}.
\]
Since \(p^{1/p}\le e\), it follows that
\[
p^{(2-r)/(rp)}
=
\bigl(p^{1/p}\bigr)^{(2-r)/r}
\le
e^{(2-r)/r},
\]
and hence
\[
\|t\|_p \le e^{(2-r)/r} p^{-1/r}.
\]
Therefore,
\[
p^{1/r}\|t\|_p \le e^{(2-r)/r}.
\]
Combining this with
\[
\sqrt p\,\|t\|_2 \le 1,
\]
we arrive at
\[
\sqrt p\,\|t\|_2 + p^{1/r}\|t\|_p
\le
1 + e^{(2-r)/r}.
\]
This proves the claim, with
\[
L(r)=1+e^{(2-r)/r}.
\]
\end{proof}

To use Theorem \ref{thm:HMO} with the ranks $p_k=\max\{2,\log(k+2)\}$, we also need a uniform way to absorb $1\le p\le 2$.

\begin{lemma}\label{lem:low-order-absorption}
There exists a constant $C_r<\infty$, depending only on $r$, such that for every $u\in\R^N$ and every $1\le p\le 2$,
\begin{equation}\label{eq:low-order-absorption}
\rho_2(u)\le C_r\,\norm{X_u}_{L_p}.
\end{equation}
Consequently, if a sequence $(y_k)_{k\ge 1}$ satisfies
\[
\norm{X_{y_k}}_{L_{\log(k+2)}}\le B
\qquad\text{for all }k\ge 1,
\]
then
\[
\rho_{p_k}(y_k)\le C_r B,
\qquad p_k=\max\{2,\log(k+2)\},
\qquad k\ge 1.
\]
\end{lemma}

\begin{proof}
Let $m_q:=\textsf{E}|X_1|^q$. Since $\textsf{P}(|X_1|>t)=e^{-t^r}$, integration by parts gives
\[
m_q=q\int_0^\infty t^{q-1}e^{-t^r}\,dt=\Gamma\!\left(1+\frac{q}{r}\right),
\]
so in particular $m_2,m_4<\infty$.

 Since the $X_i$ are symmetric, they are centered, hence
\[
\textsf{E} X_u^2 = m_2\sum_{i=1}^N u_i^2 = m_2\norm{u}_2^2.
\]
Expanding the fourth power and using independence together with symmetry, only the pairings survive:
\[
\textsf{E} X_u^4
=
 m_4\sum_{i=1}^N u_i^4 + 6m_2^2\sum_{1\le i<j\le N} u_i^2u_j^2
\le C_r\Bigl(\sum_{i=1}^N u_i^2\Bigr)^2
= C_r\norm{u}_2^4.
\]
Now apply H\"older's inequality in the form
\[
\textsf{E} |X_u|^2 = \textsf{E}\bigl(|X_u|^{2/3}|X_u|^{4/3}\bigr)
\le (\textsf{E}|X_u|)^{2/3}(\textsf{E}|X_u|^4)^{1/3}.
\]
Rearranging yields
\[
\textsf{E}|X_u|
\ge
\frac{(\textsf{E} X_u^2)^{3/2}}{(\textsf{E} X_u^4)^{1/2}}
\ge c_r\norm{u}_2.
\]
Since $\norm{X_u}_{L_p}\ge \norm{X_u}_{L_1}=\textsf{E}|X_u|$ for every $p\ge 1$, we obtain
\[
\norm{X_u}_{L_p}\ge c_r\norm{u}_2
\qquad (1\le p\le 2).
\]
On the other hand,
\[
\rho_2(u)=\sqrt2\,\norm{u}_2 + 2^{1/r}\norm{u}_\infty
\le (\sqrt2+2^{1/r})\norm{u}_2,
\]
because $\norm{u}_\infty\le \norm{u}_2$. This proves \eqref{eq:low-order-absorption}.

For the consequence, if $\log(k+2)\ge 2$, then Theorem \ref{thm:HMO} gives
\[
\rho_{p_k}(y_k)\le C_r\norm{X_{y_k}}_{L_{p_k}}\le C_r B.
\]
If $\log(k+2)<2$, then $p_k=2$, and the conclusion follows from \eqref{eq:low-order-absorption}. The proof is complete.
\end{proof}
The following lemma provides the Equivalence of the convex hull and the absolutely convex hull.
\begin{lemma}
Let
\[
p_k:=\max\{2,\log(k+2)\}, \qquad k\ge1.
\]
Fix a finite set \(T\subset \mathbb R^N\). Assume first that there exists a sequence \((y_k)_{k\ge1}\subset \mathbb R^N\)
such that
\[
T-T\subset \operatorname{absconv}\{y_k:k\ge1\}
\]
and
\[
\|X_{y_k}\|_{L_{p_k}}\le B, \qquad k\ge1.
\]
Then there exists another sequence \((z_m)_{m\ge1}\subset \mathbb R^N\) such
that
\[
T-T\subset \operatorname{conv}\{z_m:m\ge1\}
\]
and
\[
\|X_{z_m}\|_{L_{p_m}}\le C_r B, \qquad m\ge1,
\]
where \(C_r<\infty\) depends only on \(r\). Conversely, if there exists a sequence \((z_m)_{m\ge1}\subset \mathbb R^N\)
such that
\[
T-T\subset \operatorname{conv}\{z_m:m\ge1\}
\]
and
\[
\|X_{z_m}\|_{L_{p_m}}\le B, \qquad m\ge1,
\]
then the same sequence also satisfies
\[
T-T\subset \operatorname{absconv}\{z_m:m\ge1\}
\]
with the same bound.
Hence the formulation with the convex hull and the formulation with the
absolutely convex hull are equivalent up to a constant depending only on \(r\).
\end{lemma}

\begin{proof}
The second implication is immediate, since
\[
\operatorname{conv}\{z_m:m\ge1\}\subset \operatorname{absconv}\{z_m:m\ge1\}.
\]

To prove the converse direction, define
\[
z_{2k-1}:=y_k,
\qquad
z_{2k}:=-y_k,
\qquad k\ge1.
\]
Then
\[
\operatorname{absconv}\{y_k:k\ge1\}
=
\operatorname{conv}\{z_m:m\ge1\},
\]
so
\[
T-T\subset \operatorname{conv}\{z_m:m\ge1\}.
\]

It remains to show that the sequence \((z_m)_{m\ge1}\) satisfies the required
moment bound.
For \(m=2k-1\) or \(m=2k\), we have
\[
p_m\le p_{2k}\le 2p_k.
\]
Therefore,
\[
\rho_{p_m}(z_m)=\rho_{p_m}(y_k)\le 2^{1/r}\rho_{p_k}(y_k),
\]
because
\[
\sqrt{p_m}\le \sqrt{2}\,\sqrt{p_k}\le 2^{1/r}\sqrt{p_k},
\qquad
p_m^{1/r}\le 2^{1/r}p_k^{1/r}.
\]
Now Theorem~\ref{thm:HMO} gives
\[
\|X_{z_m}\|_{L_{p_m}}
\le C_r \rho_{p_m}(z_m)
\le C_r 2^{1/r}\rho_{p_k}(y_k)
\le C_r' \|X_{y_k}\|_{L_{p_k}}
\le C_r' B.
\]
Here \(C_r'\) depends only on \(r\).
This proves the claim.
\end{proof}

\section{Expected suprema over random sign sets}

We begin with the estimate for the expected supremum. Let
\(\varepsilon^1,\ldots,\varepsilon^M\) be independent random vectors, each distributed uniformly
on \(\{-1,1\}^d\), and set
\[
T_\varepsilon:=\{0,\varepsilon^1,\ldots,\varepsilon^M\}.
\]
The goal of this section is to prove that, whenever \(M\le e^Q\),
\[
\mathbb E_\varepsilon b_X(T_\varepsilon)\lesssim_r \sqrt{dQ}.
\]
In particular, the required upper bound for \(b_X(T_\varepsilon)\) holds with positive
probability.

\begin{lemma}\label{lem:rademacher-max}
Let $\eps^1,\dots,\eps^M$ be independent random vectors, each uniformly distributed on $\{-1,1\}^d$. Then for every fixed $x\in\R^d$,
\begin{equation}\label{eq:rademacher-max}
\textsf{E}_\eps \max_{1\le j\le M} |\ip{\eps^j}{x}|
\le C\sqrt{\log(2M)}\,\norm{x}_2,
\end{equation}
where $C$ is an absolute constant.
\end{lemma}

\begin{proof}
Fix $j$ and set $S_j:=\ip{\eps^j}{x}=\sum_{i=1}^d \eps_i^j x_i$. For every $\lambda\in\R$,
\[
\textsf{E}_\eps e^{\lambda S_j}
=
\prod_{i=1}^d \textsf{E} e^{\lambda \eps_i^j x_i}
=
\prod_{i=1}^d \cosh(\lambda x_i)
\le
\exp\!\left(\frac{\lambda^2\norm{x}_2^2}{2}\right),
\]
because $\cosh u\le e^{u^2/2}$ for all $u\in\R$. By Chernoff's bound,
\[
\textsf{P}_\eps\{|S_j|\ge t\}\le 2\exp\!\left(-\frac{t^2}{2\norm{x}_2^2}\right)
\qquad (t\ge 0).
\]
Integrating the tail, one obtains for every $s\ge 1$,
\[
\bigl(\textsf{E}_\eps |S_j|^s\bigr)^{1/s}\le C\sqrt{s}\,\norm{x}_2.
\]
Now choose $s:=\log(2M)\ge 1$. Then
\[
\textsf{E}_\eps \max_{1\le j\le M}|S_j|
\le
\Bigl(\textsf{E}_\eps \max_{1\le j\le M}|S_j|^s\Bigr)^{1/s}
\le
\Bigl(\sum_{j=1}^M \textsf{E}_\eps |S_j|^s\Bigr)^{1/s}
\le
M^{1/s} C\sqrt{s}\,\norm{x}_2.
\]
Since $M^{1/\log(2M)}\le e$, the estimate \eqref{eq:rademacher-max} follows.
\end{proof}

\begin{proposition}\label{prop:bX-random-sign-set}
Let
\[
T_\eps:=\{0,\eps^1,\dots,\eps^M\}\subset\R^d,
\]
where $\eps^1,\dots,\eps^M$ are independent uniform sign vectors. If $M\le e^Q$ and $Q\ge 1$, then
\begin{equation}\label{eq:bX-random-sign-set}
\textsf{E}_\eps\,\bx(T_\eps)\le C_r\sqrt{dQ}.
\end{equation}
Consequently,
\begin{equation}\label{eq:bX-random-sign-set-half-prob}
\textsf{P}_\eps\!\left\{\bx(T_\eps)\le 2C_r\sqrt{dQ}\right\}\ge \frac12.
\end{equation}
\end{proposition}

\begin{proof}
Fix a realization $X=(X_1,\dots,X_d)$ of the Weibull vector. Then
\[
\sup_{t\in T_\eps}X_t
=
\max\Bigl\{0,\max_{1\le j\le M} \ip{\eps^j}{X}\Bigr\}
\le
\max_{1\le j\le M}|\ip{\eps^j}{X}|.
\]
Taking conditional expectation with respect to $\eps$ and applying Lemma \ref{lem:rademacher-max},
\[
\textsf{E}_\eps\Bigl[\sup_{t\in T_\eps} X_t\,\Big|\,X\Bigr]
\le C\sqrt{\log(2M)}\,\norm{X}_2
\le C\sqrt{Q}\,\norm{X}_2.
\]
Now take expectation with respect to $X$ and use Jensen's inequality:
\[
\textsf{E}\norm{X}_2
\le
\Bigl(\textsf{E}\sum_{i=1}^d X_i^2\Bigr)^{1/2}
=
\sqrt{d\,\textsf{E} X_1^2}
\le C_r\sqrt{d}.
\]
Combining the two inequalities yields \eqref{eq:bX-random-sign-set}. Finally, \eqref{eq:bX-random-sign-set-half-prob} is an immediate consequence of Markov's inequality.
\end{proof}

\section{Small sets associated with fixed admissible sequences}

We next turn to the obstruction to an absolutely convex covering. As explained above, if
\[
T_\varepsilon\subset \operatorname{absconv}\{y_k:k\ge1\},
\]
then every point \(\varepsilon^j\) must belong to the set \(U(y_\bullet)\) associated with
\((y_k)\). We begin with the simpler case in which the sequence \((y_k)\) is fixed.

The purpose of this section is to prove that, for every fixed sequence satisfying
\[
\rho_{p_k}(y_k)\le A\sqrt{dQ},\qquad k\ge1,
\]
the set \(U(y_\bullet)\) has exponentially small measure in the discrete cube. The case in which
\((y_k)\) may depend on \(\varepsilon^1,\ldots,\varepsilon^M\) is treated in Section \(5\).

\begin{definition}\label{def:W-admissible}
Let $W>0$. A sequence $(y_k)_{k\ge 1}\subset\R^d$ is called \emph{$W$-admissible} if
\[
\rho_{p_k}(y_k)\le W,
\qquad
p_k=\max\{2,\log(k+2)\},
\qquad k\ge 1.
\]
For such a sequence we define
\begin{equation}\label{eq:U-def}
U(y_\bullet)
:=
\Bigl\{\sigma\in\{-1,1\}^d:
\sup_{k\ge 1}|\ip{y_k}{\sigma}|\ge d/2
\Bigr\}.
\end{equation}
\end{definition}

\begin{lemma}\label{lem:small-covered-measure}
Fix $A>0$, let $Q\ge 1$ and $d\ge 1$, and put
\[
W:=A\sqrt{dQ}.
\]
Then there exists a constant $c_{A,r}>0$ such that, whenever $d/Q$ is sufficiently large, every $W$-admissible sequence satisfies
\begin{equation}\label{eq:small-covered-measure}
\mu\bigl(U(y_\bullet)\bigr)
\le
\exp\!\left(-c_{A,r}\,\frac{d}{Q}\right),
\end{equation}
where $\mu$ denotes the uniform probability measure on $\{-1,1\}^d$.
\end{lemma}

\begin{proof}
Let
\[
P:=(4W)^r.
\]
Suppose first that $p_k>P$. Since $\rho_{p_k}(y_k)\le W$, the second term in \eqref{eq:gauge-def} gives
\[
\norm{y_k}_\infty \le W p_k^{-1/r} < \frac14.
\]
Hence for every $\sigma\in\{-1,1\}^d$,
\[
|\ip{y_k}{\sigma}|
\le
\sum_{i=1}^d |(y_k)_i|
\le d\norm{y_k}_\infty
<\frac d4,
\]
so such indices do not contribute to $U(y_\bullet)$.

Now consider an index with $p_k\le P$. The first term in \eqref{eq:gauge-def} yields
\[
\norm{y_k}_2\le \frac{W}{\sqrt{p_k}}.
\]
Let $\sigma=(\sigma_1,\dots,\sigma_d)$ be a uniform random sign vector. Exactly as in the proof of Lemma \ref{lem:rademacher-max},
\[
\textsf{P}\bigl\{|\ip{y_k}{\sigma}|\ge t\bigr\}
\le 2\exp\!\left(-\frac{t^2}{2\norm{y_k}_2^2}\right).
\]
Taking $t=d/2$ and using the above bound on $\norm{y_k}_2$ gives
\begin{equation}\label{eq:single-cap-bound}
\mu\Bigl\{\sigma:|\ip{y_k}{\sigma}|\ge d/2\Bigr\}
\le 2\exp\!\left(-c\,\frac{d p_k}{A^2Q}\right).
\end{equation}

Group the indices according to the value of $p_k$. For each integer $m\ge 2$, the set of indices with $m\le p_k<m+1$ has cardinality at most $e^{m+1}$, because $p_k=\max\{2,\log(k+2)\}$. Therefore, using the union bound and \eqref{eq:single-cap-bound}, we have
\[
\mu\bigl(U(y_\bullet)\bigr)
\le
C\sum_{m=2}^{\lceil P\rceil} e^m \exp\!\left(-c\,\frac{dm}{A^2Q}\right)
\le
C\sum_{m=2}^{\infty}
\exp\!\left(-m\Bigl(c\,\frac{d}{A^2Q}-1\Bigr)\right).
\]
If $d/Q$ is large enough, then the coefficient in parentheses is bounded below by $c'_{A,r} d/Q$. The series is then dominated by its first term, which proves \eqref{eq:small-covered-measure}.
\end{proof}

\begin{remark}\label{rem:fixed-covering-essential}
In Lemma \ref{lem:small-covered-measure}, the probability is taken only with respect to the test point $\sigma$, while the vectors $y_k$ are fixed in advance. If the covering sequences were allowed to depend on the tested sign vector itself, then the estimate would be false: for instance, taking $y_1=\sigma$ makes $|\ip{y_1}{\sigma}|=d$.
The role of the VC argument in the next section is precisely to overcome the fact that, in the actual counterexample, the covering sequence may depend on the entire random sample.
\end{remark}

\section{Uniform exclusion of sample-dependent admissible sequences}

The estimate proved in Section \(4\) applies to one fixed admissible sequence. In the covering
problem, however, the sequence \((y_k)\) may depend on the random points
\(\varepsilon^1,\ldots,\varepsilon^M\). Thus Lemma \(4.2\) alone does not exclude the possibility
that, after the points are given, one can choose a sequence whose associated set
\(U(y_\bullet)\) contains all of them.

The purpose of this section is to rule out this possibility. We show, using VC theory, that with
probability tending to one there is no admissible sequence \((y_k)\) such that
\[
\varepsilon^j\in U(y_\bullet),\qquad 1\le j\le M.
\]
The argument uses only that the sets \(U(y_\bullet)\) belong to a VC class of controlled
dimension.

We first recall the standard notions of range spaces, traces, growth
functions, and VC dimension.
They will be used only in a very simple way, but it is convenient to fix the
notation.

\begin{definition}\label{def:range-space}
Let $\Omega$ be a set and let $\mathcal{C}\subset 2^\Omega$ be a family of subsets. The pair $(\Omega,\mathcal{C})$ is called a \emph{range space}. For a finite subset $S\subset\Omega$, the \emph{trace} of $\mathcal{C}$ on $S$ is
\[
\mathcal{C}|_S:=\{C\cap S:C\in\mathcal{C}\}\subset 2^S.
\]
The \emph{growth function} is
\[
\Pi_{\mathcal C}(n):=\sup_{|S|=n}|\mathcal{C}|_S|,
\qquad n\ge 1.
\]
\end{definition}

\begin{definition}
A finite set $S\subset\Omega$ is \emph{shattered} by $\mathcal{C}$ if $\mathcal{C}|_S=2^S$. The \emph{VC dimension} of $\mathcal{C}$ is
\[
\VCdim(\mathcal{C})
:=
\sup\{|S|: S\subset\Omega\text{ finite and shattered by }\mathcal{C}\}.
\]
\end{definition}

\begin{theorem}[Sauer--Shelah lemma]\label{thm:sauer-shelah}
If $\VCdim(\mathcal{C})=v<\infty$, then for every $n\ge 1$,
\begin{equation}\label{eq:sauer-shelah}
\Pi_{\mathcal{C}}(n)
\le
\sum_{j=0}^v \binom{n}{j}
\le
(Cn)^v.
\end{equation}
See Theorem 8.2.16 in \cite{vershynin2018high}.
\end{theorem}
We now turn to the class of sets that is relevant for the present argument.

For our purposes, the relevant sets are unions of halfspaces on the discrete
cube \(\Omega_d=\{-1,1\}^d\).
Indeed, each exposed set is built from conditions of the form
\(|\langle y_k,\sigma\rangle|\ge d/2\), and each such condition is the union
of two halfspace conditions.
This leads naturally to the class \(\mathcal{U}_{d,L}\) defined below.

Let
\[
\Omega_d:=\{-1,1\}^d.
\]
Denote by $\mathcal{H}_d$ the class of traces on $\Omega_d$ of halfspaces in $\R^d$:
\[
\mathcal{H}_d
:=
\Bigl\{
\{\sigma\in\Omega_d: \ip{a}{\sigma}\ge b\}: a\in\R^d,\ b\in\R
\Bigr\}.
\]
For an integer $L\ge 1$, define
\[
\mathcal{U}_{d,L}
:=
\Bigl\{
H_1\cup\cdots\cup H_{2L}: H_\ell\in \mathcal{H}_d
\Bigr\}.
\]
The factor $2L$ reflects the identity
\[
|\ip{y}{\sigma}|\ge d/2
\iff
\bigl(\ip{y}{\sigma}\ge d/2\bigr)\ \text{or}\ \bigl(\ip{-y}{\sigma}\ge d/2\bigr).
\]

\begin{lemma}\label{lem:VC-union-halfspaces}
There exists an absolute constant $C$ such that
\begin{equation}\label{eq:VC-union-halfspaces}
\VCdim(\mathcal{U}_{d,L})\le C dL\log(dL+2).
\end{equation}
\end{lemma}

\begin{proof}
It is classical that halfspaces in $\R^d$ have VC dimension $d+1$; see, for instance, Matou\v sek~\cite[Chapter~10]{Matousek2002}. Therefore, by Theorem \ref{thm:sauer-shelah},
\[
\Pi_{\mathcal{H}_d}(n)
\le
(Cn)^{d+1}
\le
(Cn)^{2d}
\qquad (n\ge 1),
\]
after adjusting the absolute constant.

Now fix an $n$-point subset $S\subset\Omega_d$. Any member of $\mathcal{U}_{d,L}$ is the union of at most $2L$ traces from $\mathcal{H}_d|_S$, so
\[
|\mathcal{U}_{d,L}|_S|
\le
\bigl(\Pi_{\mathcal{H}_d}(n)\bigr)^{2L}
\le
(Cn)^{4dL}.
\]
If $S$ is shattered by $\mathcal{U}_{d,L}$, then $2^n=|\mathcal{U}_{d,L}|_S|$, hence
\[
2^n\le (Cn)^{4dL}.
\]
Set $D:=dL\ge 1$. If $n>C'D\log(D+2)$ with $C'$ sufficiently large, then
\[
n\log 2 > 4D\log(Cn),
\]
which contradicts the previous inequality. Therefore every shattered set has cardinality at most $C'D\log(D+2)$, proving \eqref{eq:VC-union-halfspaces}.
\end{proof}

\begin{remark}
A more general bound of order $O(vK\log K)$ for $K$-fold unions of a VC class of dimension $v$ is standard; see Eisenstat and Angluin~\cite{EisenstatAngluin2007}. We included the short proof above because it is exactly tailored to the present setting.
\end{remark}

We next prove a simple VC estimate at level \(1/2\).
\begin{proposition}\label{prop:epsilon-net-half}
Let $\mathcal{C}$ be a countable family of measurable subsets of a probability space $(\Omega,\mu)$, and assume $\VCdim(\mathcal{C})=v<\infty$. Let $Z_1,\dots,Z_M$ be independent $\Omega$-valued random variables with common distribution $\mu$. Then
\begin{equation}\label{eq:epsilon-net-half}
\textsf{P}\Bigl\{\exists C\in\mathcal{C}: \mu(C)\ge \tfrac12,\ \{Z_1,\dots,Z_M\}\cap C=\varnothing\Bigr\}
\le
C\,\Pi_{\mathcal C}(2M)e^{-cM}
\le C(CM)^v e^{-cM},
\end{equation}
where $c,C>0$ are absolute constants.
Consequently, if $\mathcal{D}$ is another VC class of dimension $v$ on the same space, then
\begin{equation}\label{eq:small-set-containing-sample}
\textsf{P}\Bigl\{\exists D\in\mathcal{D}: \mu(D)\le \tfrac12,\ Z_1,\dots,Z_M\in D\Bigr\}
\le C(CM)^v e^{-cM}.
\end{equation}
\end{proposition}

\begin{proof}
We first prove \eqref{eq:epsilon-net-half}. Put
\[
E:=
\left\{
\exists C\in\mathcal C:\ \mu(C)\ge \frac12,\ 
\{Z_1,\ldots,Z_M\}\cap C=\varnothing
\right\}.
\]
Let \(Z'_1,\ldots,Z'_M\) be an independent copy of
\(Z_1,\ldots,Z_M\). Define
\[
E':=
\left\{
\exists C\in\mathcal C:\ 
\{Z_1,\ldots,Z_M\}\cap C=\varnothing,\ 
\#\{j:Z'_j\in C\}\ge \frac M4
\right\}.
\]

We first compare \(E\) and \(E'\). We may assume \(M\ge 2\), since the case
\(M=1\) is absorbed by changing the absolute constants.

Since \(\mathcal C\) is countable, write
\[
\mathcal C=\{C_1,C_2,\ldots\}.
\]
For \(z=(z_1,\ldots,z_M)\in\Omega^M\), let
\[
E_k:=\left\{
z\in\Omega^M:\ \mu(C_k)\ge \frac12,\ z_i\notin C_k,\ 1\le i\le M
\right\}.
\]
Then \(E=\bigcup_{k\ge1}E_k\). For \(z\in E\), define
\[
k(z):=\min\{k:\ z\in E_k\},
\qquad
C_z:=C_{k(z)}.
\]
Thus \(C_z\in\mathcal C\), \(\mu(C_z)\ge 1/2\), and
\(z_i\notin C_z\) for all \(1\le i\le M\).

Fix \(z\in E\), and put \(p_z:=\mu(C_z)\). For
\(z'=(z'_1,\ldots,z'_M)\in\Omega^M\), define
\[
N_z(z'):=\#\{j:\ z'_j\in C_z\}.
\]
With respect to the product measure \(\mu^M\) in the variable \(z'\),
\(N_z\) has binomial distribution with parameters \(M\) and
\(p_z\ge 1/2\). Hence
\[
\textsf E_{\mu^M}N_z=Mp_z\ge \frac M2
\]
and
\[
\textsf E_{\mu^M}N_z^2
=
Mp_z(1-p_z)+(Mp_z)^2
\le Mp_z+(Mp_z)^2
\le 2(Mp_z)^2,
\]
where the last inequality uses \(M\ge2\) and \(p_z\ge1/2\).
By the Paley--Zygmund inequality,
\[
\mu^M\left\{
z'\in\Omega^M:\ N_z(z')\ge \frac12 Mp_z
\right\}
\ge c_0
\]
for some absolute constant \(c_0>0\). Since \(p_z\ge1/2\), this implies
\[
\mu^M\left\{
z'\in\Omega^M:\ N_z(z')\ge \frac M4
\right\}
\ge c_0.
\]

Now consider the subset \(B\subset\Omega^M\times\Omega^M\) given by
\[
B:=
\left\{
(z,z'):\ z\in E,\ 
\#\{j:\ z'_j\in C_z\}\ge \frac M4
\right\}.
\]
For every \((z,z')\in B\), the set \(C_z\) satisfies the defining conditions
of \( E'\).
Therefore
\[
B\subset E'.
\]
Consequently, by Fubini's theorem,
\[
\textsf P(E')
\ge
(\mu^M\otimes\mu^M)(B)
=
\int_E
\mu^M\left\{
z'\in\Omega^M:\ 
\#\{j:\ z'_j\in C_z\}\ge \frac M4
\right\}
\,d\mu^M(z).
\]
Using the lower bound above for each \(z\in E\), we get
\[
\textsf P(E')
\ge
\int_E c_0\,d\mu^M(z)
=
c_0\mu^M(E)
=
c_0\textsf P(E).
\]
Thus
\[
\textsf P(E)\le c_0^{-1}\textsf P(E').
\]

It remains to estimate \(\textsf P(E')\). Let
\(Y_1,\ldots,Y_{2M}\) be independent random variables with distribution
\(\mu\). Independently of them, choose a set
\(I\subset\{1,\ldots,2M\}\) uniformly among all subsets of cardinality \(M\).
Then the two random families
\[
(Y_i)_{i\in I}
\qquad\text{and}\qquad
(Y_i)_{i\notin I}
\]
have the same joint distribution, up to the order of the points, as
\[
(Z_1,\ldots,Z_M)
\qquad\text{and}\qquad
(Z'_1,\ldots,Z'_M).
\]
Thus we may estimate \(E'\) in this equivalent model.

Fix the values of \(Y_1,\ldots,Y_{2M}\). The only remaining randomness is the
choice of \(I\). For \(C\in\mathcal C\), define
\[
A_C:=\{i\le 2M:\ Y_i\in C\}.
\]
The number of distinct sets \(A_C\) is at most \(\Pi_{\mathcal C}(2M)\).
If \(E'\) occurs, then for some \(C\in\mathcal C\),
\[
I\cap A_C=\varnothing
\]
and
\[
\#(A_C\cap I^c)\ge \frac M4.
\]
Since \(I\cap A_C=\varnothing\), the second condition implies
\[
|A_C|\ge \frac M4.
\]
Hence, conditionally on \(Y_1,\ldots,Y_{2M}\), the event \(E'\) is contained
in the union over all traces \(A\) with \(|A|\ge M/4\) of the events
\[
I\cap A=\varnothing.
\]

Fix \(y=(y_1,\ldots,y_{2M})\in\Omega^{2M}\). For \(C\in\mathcal C\), set
\[
A_C(y):=\{i\le 2M:y_i\in C\}.
\]
The number of distinct traces \(A_C(y)\) is at most \(\Pi_{\mathcal C}(2M)\).

Let
\[
\mathcal I_M:=\{I\subset\{1,\ldots,2M\}: |I|=M\}.
\]
For this fixed \(y\), denote by \(\mathcal B(y)\subset\mathcal I_M\) the set
of choices of \(I\) for which the event \(E'\) occurs. If \(I\in\mathcal B(y)\),
then there exists \(C\in\mathcal C\) such that
\[
I\cap A_C(y)=\varnothing
\qquad\text{and}\qquad
|A_C(y)|\ge \frac M4.
\]
Hence
\[
\mathcal B(y)
\subset
\bigcup_{\substack{A\in\{A_C(y):C\in\mathcal C\}\\ |A|\ge M/4}}
\{I\in\mathcal I_M:I\cap A=\varnothing\}.
\]
For a fixed \(A\subset\{1,\ldots,2M\}\) with \(|A|=s\ge M/4\), the number of
sets \(I\in\mathcal I_M\) avoiding \(A\) is
\[
\binom{2M-s}{M},
\]
with the convention that this binomial coefficient is zero if \(2M-s<M\).
Therefore
\[
\frac{\#\{I\in\mathcal I_M:I\cap A=\varnothing\}}{\#\mathcal I_M}
=
\frac{\binom{2M-s}{M}}{\binom{2M}{M}}
\le 2^{-s}
\le 2^{-M/4}
\le e^{-cM}.
\]
Consequently,
\[
\frac{\#\mathcal B(y)}{\#\mathcal I_M}
\le
\Pi_{\mathcal C}(2M)e^{-cM}.
\]
Averaging this bound over \(y=(Y_1,\ldots,Y_{2M})\) gives
\[
\textsf P(E')
\le
\Pi_{\mathcal C}(2M)e^{-cM}.
\]
Combining this estimate with the comparison between \(E\) and \(E'\), we get
\[
\textsf P(E)
\le
C\Pi_{\mathcal C}(2M)e^{-cM}.
\]
Finally, by the Sauer--Shelah lemma,
\[
\Pi_{\mathcal C}(2M)\le (CM)^v.
\]
Thus
\[
\textsf P(E)
\le
C(CM)^v e^{-cM},
\]
which proves \eqref{eq:epsilon-net-half}.

To prove \eqref{eq:small-set-containing-sample}, apply
\eqref{eq:epsilon-net-half} to the complement class
\[
\mathcal C:=\{\Omega\setminus D:\ D\in\mathcal D\}.
\]
Taking complements preserves the growth function and hence the VC dimension.
Moreover, the event
\[
\exists D\in\mathcal D:\ \mu(D)\le \frac12,\ Z_1,\ldots,Z_M\in D
\]
is the same as
\[
\exists C\in\mathcal C:\ \mu(C)\ge \frac12,\ 
\{Z_1,\ldots,Z_M\}\cap C=\varnothing.
\]
Therefore \eqref{eq:small-set-containing-sample} follows from
\eqref{eq:epsilon-net-half}. The proof is complete.
\end{proof}

\begin{remark}
This is a standard fixed-$\varepsilon$ consequence of the Haussler--Welzl theory of epsilon-nets~\cite{HausslerWelzl1987}. The self-contained proof above is included because it isolates exactly the form needed in the present argument.
\end{remark}

We can now return to the covering problem.
The previous proposition gives a uniform bound over a VC class, and
Lemma \ref{lem:VC-union-halfspaces} shows that the exposed sets coming from admissible sequences belong
to such a class.
Combining these two facts yields the required exclusion statement.

\begin{proposition}\label{prop:no-sample-dependent-cover}
Fix $A>0$ and $1<a<2/r-1$. Let
\[
d=\lfloor Q^a\rfloor,
\qquad
M=\lfloor e^Q\rfloor,
\qquad
W=A\sqrt{dQ}.
\]
Let $\eps^1,\dots,\eps^M$ be independent uniform random vectors in $\{-1,1\}^d$. Then, as $Q\to\infty$,
\begin{equation}\label{eq:no-sample-dependent-cover}
\textsf{P}\Bigl\{\exists\text{ a $W$-admissible sequence }(y_k)_{k\ge 1}\text{ with }\eps^j\in U(y_\bullet)\ \forall 1\le j\le M\Bigr\}\longrightarrow 0.
\end{equation}
\end{proposition}

\begin{proof}
Let
\[
R:=(4W)^r,
\qquad
L:=\lceil e^R\rceil.
\]
By Lemma \ref{lem:small-covered-measure}, every fixed $W$-admissible sequence satisfies
\[
\mu\bigl(U(y_\bullet)\bigr)
\le
\exp\!\left(-c_{A,r}\,\frac{d}{Q}\right)
<\frac12
\]
for all sufficiently large $Q$, because $d/Q=Q^{a-1}(1+o(1))\to\infty$.

Moreover, the proof of Lemma \ref{lem:small-covered-measure} shows that indices with $p_k>R$ do not contribute to $U(y_\bullet)$. Thus $U(y_\bullet)$ is determined entirely by those indices with $p_k\le R$, and the number of such indices is at most $e^R\le L$. Since each condition $|\ip{y_k}{\sigma}|\ge d/2$ is the union of two halfspace events, it follows that
\[
U(y_\bullet)\in \mathcal U_{d,L}.
\]
Set
\[
v:=\VCdim(\mathcal U_{d,L}).
\]
By Lemma \ref{lem:VC-union-halfspaces},
\begin{equation}\label{eq:v-bound}
v\le C dL\log(dL+2).
\end{equation}

If there exists a \(W\)-admissible sequence \((y_k)_{k\ge1}\), possibly depending
on \(\eps^1,\dots,\eps^M\), such that
\[
\eps^j\in U(y_\bullet), \qquad 1\le j\le M,
\]
then there exists a set \(U\in \mathcal{U}_{d,L}\) such that
\[
\mu(U)\le \frac12,
\qquad
\eps^1,\dots,\eps^M\in U.
\]
Hence, by Proposition \ref{prop:epsilon-net-half} in the form \eqref{eq:small-set-containing-sample}, the probability of this event is at most
\begin{equation}\label{eq:uniform-bound-raw}
C(CM)^v e^{-cM}.
\end{equation}

It remains to show that the right-hand side tends to zero. Since
\[
R=(4A\sqrt{dQ})^r \le C_A (dQ)^{r/2}\le C_A Q^{\beta},
\qquad
\beta:=\frac{r(a+1)}{2},
\]
and $a<2/r-1$, we have $\beta<1$. Therefore
\[
R=O_A(Q^{\beta})=o(Q),
\qquad
L=e^{o(Q)}.
\]
Because $d=Q^a(1+o(1))$ is polynomial in $Q$, \eqref{eq:v-bound} yields
\[
v\log(CM)
\le
C dL\log(dL+2)\cdot Q
=
 e^{o(Q)}.
\]
On the other hand,
\[
M=e^{Q(1+o(1))}.
\]
Consequently,
\[
\log\bigl(C(CM)^v e^{-cM}\bigr)
\le
 e^{o(Q)} - c e^{Q(1+o(1))}
\longrightarrow -\infty.
\]
Thus \eqref{eq:uniform-bound-raw} tends to zero, which proves \eqref{eq:no-sample-dependent-cover}.
\end{proof}

\section{Construction of the counterexample and proof of the main theorem}

We now combine Proposition \ref{prop:bX-random-sign-set} and
Proposition \ref{prop:no-sample-dependent-cover}.

\begin{proof}[Proof of Theorem \ref{thm:quantitative-counterexample}]
Fix $A>0$ and $1<a<2/r-1$. Let
\[
d=\lfloor Q^a\rfloor,
\qquad
M=\lfloor e^Q\rfloor,
\qquad
T_\eps=\{0,\eps^1,\dots,\eps^M\}\subset \{0\}\cup\{-1,1\}^d.
\]
By Proposition \ref{prop:bX-random-sign-set}, the event
\[
E_1:=\Bigl\{\bx(T_\eps)\le C_r\sqrt{dQ}\Bigr\}
\]
has probability bounded below by a positive absolute constant (for instance, by $1/2$ after enlarging $C_r$ if necessary). By Proposition \ref{prop:no-sample-dependent-cover}, the event
\[
E_2:=\Bigl\{\text{there is no $A\sqrt{dQ}$-admissible sequence $(y_k)$ with }\eps^j\in U(y_\bullet)\ \forall j\Bigr\}
\]
has probability tending to $1$ as $Q\to\infty$. Hence, for all sufficiently large $Q$, the intersection $E_1\cap E_2$ has positive probability.
Choose deterministic vectors \(\varepsilon^1,\ldots,\varepsilon^M\) in this intersection, and define
\[
T:=\{0,\eps^1,\dots,\eps^M\}.
\]
Then \eqref{eq:bX-upper-main} holds.

Assume, toward a contradiction, that there exists a sequence $(y_k)_{k\ge 1}\subset\R^d$ satisfying \eqref{eq:T-covered} and \eqref{eq:gauge-controlled}. Since each $\eps^j\in T\subset \absconv\{y_k\}$, identity \eqref{eq:support-absconv} with $\theta=\eps^j$ gives
\[
d
=
\ip{\eps^j}{\eps^j}
\le
\sup_{k\ge 1}|\ip{\eps^j}{y_k}|.
\]
In particular,
\[
\eps^j\in U(y_\bullet),
\qquad 1\le j\le M.
\]
But \eqref{eq:gauge-controlled} says exactly that $(y_k)$ is $A\sqrt{dQ}$-admissible, contradicting the definition of $E_2$. Therefore no such sequence exists, and Theorem \ref{thm:quantitative-counterexample} is proved.
\end{proof}

\begin{proof}[Proof of Theorem \ref{thm:main-impossibility}]
Assume for contradiction that a universal constant $C_r$ exists such that \eqref{eq:universal-false-statement} holds for every finite set $T\subset\R^N$.
Fix any
\[
1<a<\frac{2}{r}-1.
\]
By Theorem \ref{thm:quantitative-counterexample}, for every sufficiently large $Q$ there exists a set $T\subset\R^d$ with
\[
\bx(T)\le C_r'\sqrt{dQ},
\]
but admitting no sequence $(y_k)$ that both covers $T$ and satisfies
\[
\rho_{p_k}(y_k)\le A\sqrt{dQ},
\qquad k\ge 1,
\]
for the fixed constant $A$ chosen below.

On the other hand, the assumed universal principle yields a free rank-wise covering sequence $(y_k)$ for this $T$ such that
\[
\norm{X_{y_k}}_{L_{\log(k+2)}}\le C_r\,\bx(T),
\qquad k\ge 1.
\]
By Lemma \ref{lem:low-order-absorption},
\[
\rho_{p_k}(y_k)
\le
C_r''\norm{X_{y_k}}_{L_{\log(k+2)}}
\le
C_r'' C_r\,\bx(T)
\le
C_r'' C_r C_r'\sqrt{dQ}.
\]
Thus the covering sequence satisfies the gauge bound with
\[
A:=C_r'' C_r C_r'.
\]
Since the same covering sequence also covers $T-T$, and therefore covers $T$, this contradicts Theorem \ref{thm:quantitative-counterexample}. The contradiction proves Theorem \ref{thm:main-impossibility}.
\end{proof}

\textbf{Acknowledgment}. The work of Hanchao Wang (corresponding author) was supported by the National Key R\&D Program of China (No.2024YFA1013501), the National Natural Science Foundation of China (No. 12571162), and Shandong Provincial Natural Science Foundation (No. ZR2024MA082).

\printbibliography
\end{document}